\documentclass[a4paper, reqno, 11pt]{amsart}
\pdfoutput=1
\usepackage[utf8]{inputenc}
\usepackage[T1]{fontenc}
\usepackage{amssymb,mathtools}
\usepackage{mathrsfs}
\usepackage{amstext}
\usepackage{amsfonts}
\usepackage{latexsym}
\usepackage{fancyref}
\usepackage{comment}
\usepackage[final]{pdfpages}
\usepackage{fancyhdr}
\usepackage{xr} 
\usepackage{textcomp}
\usepackage{subfiles}
\usepackage{amsthm}
\usepackage{bbm}
\usepackage{dsfont}
\usepackage{tabularx}
\usepackage{setspace}
\usepackage{lmodern}
\usepackage{longtable}
\usepackage{geometry}
\usepackage[all]{xy}
\usepackage{microtype}
\usepackage{afterpage}
\usepackage[normalem]{ulem}
\usepackage{enumitem}
\usepackage{tikz-cd}
\usepackage[all]{xy}
\usepackage{microtype}
\usepackage{enumitem}
\usepackage{extarrows}
\setlist[enumerate,1]{label=\textup{(\arabic*)}}

\usepackage[lite]{amsrefs}

\usepackage{mathabx,epsfig}

\renewcommand*{\PrintDOI}[1]{\href{http://dx.doi.org/\detokenize{#1}}{doi: \detokenize{#1}}}

\BibSpec{book}{%
  +{}  {\PrintPrimary}                {transition}
  +{,} { \textit}                     {title}
  +{.} { }                            {part}
  +{:} { \textit}                     {subtitle}
  +{,} { \PrintEdition}               {edition}
  +{}  { \PrintEditorsB}              {editor}
  +{,} { \PrintTranslatorsC}          {translator}
  +{,} { \PrintContributions}         {contribution}
  +{,} { }                            {series}
  +{,} { \voltext}                    {volume}
  +{,} { }                            {publisher}
  +{,} { }                            {organization}
  +{,} { }                            {address}
  +{,} { \PrintDateB}                 {date}
  +{,} { }                            {status}
  +{}  { \parenthesize}               {language}
  +{}  { \PrintTranslation}           {translation}
  +{;} { \PrintReprint}               {reprint}
  +{.} { }                            {note}
  +{.} {}                             {transition}
  +{} { \PrintDOI}                   {doi}
  +{} { available at \url}            {eprint}
  +{}  {\SentenceSpace \PrintReviews} {review}
}

\usepackage[pdftitle={Relatively Topologically Principal},
pdfauthor={EGL},
pdfsubject={Mathematics}
]{hyperref}

\newtheorem{thm}[equation]{Theorem}

\newtheorem{prop}[equation]{Proposition}

\newtheorem{lem}[equation]{Lemma}

\newtheorem{cor}[equation]{Corollary}

\theoremstyle{definition}

\newtheorem{defn}[equation]{Definition}

\theoremstyle{remark}
\newtheorem{remark}[equation]{Remark}
\newtheorem{rem}[equation]{Remark}

\newtheorem{example}[equation]{Example}

\newcommand{\thmref}[1]{Theorem~\ref{#1}}
\newcommand{\secref}[1]{Section~\ref{#1}}
\newcommand{\proref}[1]{Proposition~\ref{#1}}
\newcommand{\lemref}[1]{Lemma~\ref{#1}}
\newcommand{\corref}[1]{Corollary~\ref{#1}}

\numberwithin{equation}{section}

\newcommand{\C}{\mathbb{C}}
\newcommand{\N}{\mathbb{N}}
\newcommand{\Q}{\mathbb{Q}}

\newcommand{\Z}{\mathbb{Z}}

\newcommand{\dom}[1]{\operatorname{dom}(#1)}

\newcommand{\norm}[1]{\left\lVert#1\right\rVert}

\newcommand{\ssrc}[1]{\textbf{s}(#1)}

\newcommand{\ttrg}[1]{\textbf{t}(#1)}

\def\1{\mathbbm 1}

\newcommand{\Toepr}{\mathcal{T}_\lambda}

\def\nx{\N^\times}
\def\nxxz{\nx \ltimes \Z}
\def\inv{^{-1}}

\def\nxxzn{\nx \ltimes \Z/n\Z}
\def\nxxqz{\nx \ltimes \Q/\Z}

\newcommand{\mc}[1]{\mathcal{#1}}
\newcommand{\mbb}[1]{{\mathbf{#1}}}
\newcommand{\msc}[1]{\mathscr{#1}}

\author[C. J. Eagle]{Christopher J. Eagle${}^1$} 
\address[C. J. Eagle]
{University of Victoria, Department of Mathematics and Statistics, PO BOX 1700 STN CSC, Victoria, British Columbia, Canada, V8W 2Y2}%
\email{eaglec@uvic.ca}
\urladdr{http://www.math.uvic.ca/~eaglec}
\thanks{${}^1$ Supported by NSERC Discovery Grant RGPIN-2021-02459}

\author[G. Goerke]{Gavin Goerke}
\address[G. Goerke]{University of Victoria, Department of Mathematics and Statistics, PO BOX 1700 STN CSC, Victoria, British Columbia, Canada, V8W 2Y2}%
\email{gavingoerke@icloud.com}

\author[M. Laca]{Marcelo Laca${}^2$}
\address[M. Laca]{Department of Mathematics and Statistics, University of Victoria, Victoria, P.O Box 1700 STN CSC, BC V8W 2Y2, Canada}
\email{laca@uvic.ca}
\thanks{${}^2$ Supported by NSERC Discovery Grant RGPIN-2017-04052}
\date{\today}

\title[Relative topological principality and ideal intersection property]{Relative topological principality and the ideal intersection property for groupoid C*-algebras}

\begin{document}

\begin{abstract}
 We introduce the notion of relative topological principality for a family $\{H_\alpha\}$ of open subgroupoids of a Hausdorff \'etale groupoid $G$. The  C*-algebras $C^*_r(H_\alpha)$ of the groupoids $H_\alpha$ embed in  $ C^*_r(G)$ and we show that if $G$ is topologically principal relative to $\{H_\alpha\}$  then a representation of $C^*_r(G)$ is faithful if and only if its restriction to each of the subalgebras $C^*_r(H_\alpha)$ is faithful. This variant of the ideal intersection property potentially involves several subalgebras, and gives a new method of verifying injectivity of representations of reduced groupoid C*-algebras. As applications we prove a uniqueness theorem for Toeplitz C*-algebras of left cancellative small categories  that generalizes a recent result of Laca and  Sehnem for Toeplitz algebras of group-embeddable monoids, and we also discuss and compare concrete examples arising from integer arithmetic.
\end{abstract}

\maketitle

\section{Introduction}
One often wants a method for determining the injectivity of $*$-homorphisms with domain a fixed C*-algebra $A$.  In some situations there may be a C*-subalgebra $B$ of $A$ such that any $*$-homomorphism is injective on $A$ if and only if it is injective on $B$; this is particularly useful if the subalgebra $B$ is more easily understood than the whole algebra $A$.  Equivalently, we say that a C*-subalgebra $B$ of $A$  \textit{detects ideals} in $A$, or has \textit{the ideal intersection property in $A$}, if the only (closed 2-sided) ideal $J$ of $A$ such that $B \cap J = \{0\}$ is the zero ideal. If $B$ detects ideals in $A$, then we can verify injectivity of a $*$-homomorphism on $A$ by checking that it is injective on  $B$.  There are many cases known where certain types of C*-algebras have important subalgebras with the ideal intersection property, e.g. \cites{KawamuraTomiyama,ArchboldSpielberg,ExelLacaQuigg,exelNonHausdorffEtaleGroupoids2011,LS,brown2015cartan,Starling,kennedyIdealIntersectionProperty2021} (see  \secref{sec:IdealProperty} below). 

In this paper we consider the ideal intersection property for subalgebras and, more generally, families of subalgebras of $C^*_r(G)$, where $G$ is a 
Hausdorff \'etale groupoid.  We introduce the notion of a groupoid being \emph{relatively topologically principal} with respect to a family of open subgroupoids.  Our main result is that if $G$ is relatively topologically principal with respect to a family $\mathscr{H}$ of open subgroupoids then a representation of $C^*_r(G)$ is faithful if and only if it is faithful on $C^*_r(H)$ for each $H \in \mathscr{H}$.  Our notion of relative topological principality is a generalization of effectiveness or topological principality for \'etale groupoids, but our main result also captures several cases of the ideal intersection property that had been previously studied outside of the context of effective groupoids, including results from \cite{brown2015cartan}, \cite{LS}, and \cite{Starling}.

The paper is structured as follows.  \secref{sec:preliminaries} consists of preliminaries on inverse semigroups, \'etale groupoids, and the C*-algebras associated to these objects.   In \secref{sec:IdealProperty} we review several previous results about the ideal intersection property in order to set the context for the following section.  In \secref{sec:reltopprin} we give the definition of relative topological principality and prove our main result, \thmref{thm:groupoidUniquenessThm}.  We conclude by illustrating our result in two kinds of examples: In \secref{sec:smallcats} we consider the C*-algebras associated by Spielberg to left cancellative small categories \cite{SpielbergSmallCats}, and in \secref{sec:integerarith} we discuss concrete examples from integer arithmetic introduced in \cite{Laca-Schulz23}.

 \textbf{Acknowledgements:}  The results in this paper form part of the the second named author's MSc~Thesis \cite{GG-MSCthesis2023} written under the supervision of the first and third named authors. We would like to thank Jack Spielberg for several helpful comments on the thesis.  We would also like to thank Sergey Neshveyev and Bartosz Kwa\'{s}niewski for independently noticing that the second countability assumption made in the first posted version of this paper was not needed for our results, and in fact not used in our proofs; consequently this assumption has been removed. In addition, we also thank Kwa\'{s}niewski  for telling us how to use \cite{kwa-mey} to remove second countability from the second example at the end of \secref{sec:IdealProperty}, and Neshveyev for pointing out that the theory of isotropy fibers of ideals recently developed in \cite{Chr-Nes-Isotropy-2023} is particularly well-suited to prove results such as those obtained here, including our \thmref{thm:groupoidUniquenessThm}.
 
 {We would like to thank the anonymous referee for several helpful comments and suggestions.}
   
\section{Preliminaries}\label{sec:preliminaries}

By an \'etale groupoid $G$ we mean a topological groupoid with locally compact Hausdorff unit space, such that the source and target maps $\textbf{s}, \textbf{t}:G \rightrightarrows G^{(0)}$ are local homeomorphisms. A \textit{bisection} is a subset $A \subseteq G$ such that the source and target maps restrict to homeomorphisms $A \xrightarrow{\cong} \textbf{s}(A)$ and $A \xrightarrow{\cong} \textbf{t}(A)$ respectively. Recall that an \textit{inverse semigroup} is a semigroup $S$ such that for each $s \in S$ there is a unique element $s^{-1}$, the inverse of $s$, satisfying the identities $ss^{-1}s = s$ and $s^{-1}ss^{-1} = s^{-1}$. The collection of open bisections of an \'etale groupoid $G$ is denoted $\operatorname{Bis}(G)$ and forms an inverse semigroup under elementwise multiplication and inversion of subsets:
$$U\cdot V := \{\alpha\beta \; | \; \alpha \in U, \beta \in V, (\alpha,\beta) \in G^{(2)}\}, \; \; \; \; U^{-1} = \{\alpha^{-1} \; | \; \alpha \in U \}.$$

Given a unit $x \in G^{(0)}$ the \textit{isotropy} at $x$ is the group $G^x_x = \{\alpha \in G \; | \; \ssrc{\alpha} = x = \ttrg{\alpha} \}$. The \textit{isotropy} of $G$ is the subgroupoid $\operatorname{Iso}(G) := \bigcup_{x \in G^{(0)}}G^x_x$. We denote by $\operatorname{Iso}(G)^\circ$ its topological interior.

We use the standard groupoid notation for more general small categories as well. Accordingly, if $\mc{C}$ is a small category, we write $\mc{C}_x^y$ to denote the set of morphisms from $x$ to $y$. The set of objects of $\mc{C}$ is denoted $\mc{C}^{(0)}$ and is identified with the collection of identities in the set of arrows $\mc{C}^{(1)}$. We typically suppress the superscript on the set of arrows and just write $\mc{C} = \mc{C}^{(1)}$. If $a \in \mc{C}_y^z$ and $b \in \mc{C}_x^y$ we write the composition of $a$ and $b$ as $ab$.

\section{Ideal property and uniqueness theorems}\label{sec:IdealProperty}

An important early result in the study of the ideal intersection property appears in the work of Kawamura and Tomiyama \cite{KawamuraTomiyama}, where they show that when a discrete group $G$ acts on a space $X$, if the diagonal subalgebra $C(X)$ in the crossed product $C(X) \rtimes G$ has the ideal intersection property in $C(X) \rtimes G$ then the interior of the set of fixed points of $X$ for each nontrivial element of $G$ is empty; moreover, if $G$ is amenable then the converse holds. In \cite{ArchboldSpielberg} Archbold and Spielberg generalize this result to the setting of discrete C*-dynamical systems where $C(X)$ is replaced by a not necessarily commutative C*-algebra $A$. They introduced the notion of a \textit{topologically free} action, which is the appropriate notion of freeness in this setting, and showed that topological freeness of the action implies that any ideal $I \subseteq A \rtimes G$ that intersects trivially with $A$ must be contained in the kernel of the canonical quotient map $\lambda: A \rtimes G \rightarrow A \rtimes_r G$ from the full crossed product to the reduced one. In particular, when $G$ is amenable (so that $\lambda$ is injective), $A$ has the ideal intersection property in $A \rtimes G$. The result of Archbold and Spielberg indicates that if one wants to verify injectivity concretely, then one should look at the reduced crossed product so as not to be concerned with the ideal $\ker\lambda$. 

Assisted by this intuition, Exel, Laca, and Quigg \cite{ExelLacaQuigg} generalize the result of Kawamura and Tomiyama on a different axis than that of Archbold and Spielberg. They show that if a discrete group $G$ has a topologically free \textit{partial action} on a space $X$ then $C_0(X)$ has the ideal intersection property in the reduced product $C_0(X) \rtimes_r G$.  Under a second countability assumption this result has been generalized nicely to C*-algebras of Hausdorff \'etale groupoids in \cite{exelNonHausdorffEtaleGroupoids2011}. For general \'etale groupoids the concept of topological freeness splits into several non-equivalent notions; these agree if the groupoid is second-countable and Hausdorff. Perhaps the most commonly used condition in the  second-countable Hausdorff setting is the property of being \textit{topologically principal} or \textit{effective} which says that the interior of the isotropy is just the unit space.
 
 One thing all the above examples have in common is that they primarily study the ideal intersection property for an inclusion $B \subseteq A$ where $B$ is the unit fiber of a Fell bundle that satisfies some sort of topological freeness condition.  The first step away from the topologically free setting occurs in \cite[Theorem~3.1]{brown2015cartan}, where Brown, Nagy, Reznikoff, Sims, and Williams show that if $G$ is a second-countable Hausdorff \'etale groupoid, which they do not require to be topologically principal (i.e. one may have $G^{(0)} \subsetneq \operatorname{Iso}(G)^\circ$), then $C^*_r(\operatorname{Iso}(G)^\circ)$ has the ideal intersection property in $C^*_r(G)$.  This reduces to the previous case in the event $G$ is topologically principal. 
 
  The  theorem of Brown, Nagy, Reznikoff, Sims, and Williams  represents a conceptual step forward for studying ideals in $C^*_r(G)$ in the abstract, but using it in concrete situations is not necessarily easy. Moreover, there are situations where one has an open subgroupoid $H$ of an \'etale groupoid $G$ that does not contain $\operatorname{Iso}(G)^\circ$ yet  $C^*_r(H)$ still detects ideals in $C^*_r(G)$. 
  
  A very simple example of this is provided by the transformation groupoid associated to the action of a countable discrete group $G$ on itself by translation. The action is free so the interior of the isotropy is simply the unit space, hence the faithfulness criterion from \cite{ArchboldSpielberg,brown2015cartan} says that a representation of $C_0(G) \rtimes_{r} G$ is faithful if and only if it is faithful on $C_0(G)$. But this does not take into account that the action is  also minimal, and hence $C_0(G) \rtimes_{r} G \cong \mc{K}(\ell^2(G))$ is simple. With this in mind, we may simply take a single point, say, $e\in G$, for which the isotropy is $\{e\} \subset G$, and use instead the C*-algebra  of the singleton subgroupoid $\{e\} \ltimes \{e\}$, namely $\C \delta_e$, which is much smaller than $C_0(G)$ yet is equally effective at detecting ideals in $C_0(G) \rtimes_{r} G$ (admittedly because there is only one nontrivial ideal to detect in this setting). 
  
  Less obvious examples have been studied using several different approaches. For instance, in \cite{quasilat} it is first shown that the diagonal subalgebra has the ideal intersection property, but  then faithfulness is further reduced to the nonvanishing of a collection of projections \cite[Theorem 3.7]{quasilat}.  If the  quasi-lattice semigroup is finitely generated, cf. \cite[Section 6.3]{nica},
  then  a single projection, hence a $1$-dimensional subalgebra, does the job.   A second,  more recent example is the uniqueness theorem \cite[Theorem 5.1]{LS} valid for the reduced Toeplitz C*-algebra $\mathcal T_r(P)$ of a group embeddable monoid $P\subset G$. One can write $\mathcal T_r(P)$ as the C*-algebra of a partial transformation groupoid $G \ltimes \Omega$. The group $P^*$ of invertible elements of $P$ is a subgroup of $G$ and the transformation subgroupoid $P^* \ltimes \Omega$ embeds as a clopen subgroupoid of $G \ltimes \Omega$. According to \cite[Theorem~5.1]{LS},  $C^*_r(P^* \ltimes \Omega)$ detects ideals in $C^*_r(G \ltimes \Omega)$. In many cases $P^* \ltimes \Omega$ is  strictly smaller than $(\operatorname{Iso} G \ltimes \Omega)^\circ$.

  In hindsight, the results described above can be reformulated as situations in which $C^*_r(H)$ detects ideals in $C^*_r(G)$, where $H$ is an open subgroupoid of $G$.  In particular, $C^*_r(H)$ has the ideal intersection property in $C^*_r(G)$ in the following cases: 
\begin{itemize}
      \item{$G$ is the action groupoid of a discrete group acting (or partially acting) topologically freely on a locally compact Hausdorff space and $H$ is the unit space, see  \cite{ExelLacaQuigg,KawamuraTomiyama}.}
     \smallskip \item{$G$ is a topologically principal 
      Hausdorff \'etale groupoid and $H$ is the unit space, see \cite{exelNonHausdorffEtaleGroupoids2011} for  second countable $G$, and \cite[Theorem 7.29]{kwa-mey} for general $G$.} 
      
      \smallskip\item{$G$ is any second-countable Hausdorff \'etale groupoid and $H = \operatorname{Iso}{(G)^o}$,  \cite[Theorem 3.1]{brown2015cartan}.   { See \cite[Theorem 2.1]{Starling} for a generalization to more
general open subgroupoids of $\operatorname{Iso}{(G)^o}$, and \cite[Theorem 4.8]{Chr-Nes-Isotropy-2023} for a generalization that includes non-Hausdorff groupoids that admit a countable cover of open bisections.}}
     \smallskip \item{$G$ is the partial transformation groupoid associated to a group embeddable monoid $P$ and $H$ is the transformation groupoid obtained by restricting the action to the group of invertible elements of $P$, see \cite[Theorem 5.1]{LS}.}
  \end{itemize}
  
\section{Relative topological principality}\label{sec:reltopprin}
 With an eye on exploiting the orbit structure in the unit space of a groupoid in order to simplify ideal detection, we now introduce our main concept.
	
\begin{defn}\label{RelTPDef}
 	Let $G$ be an \'etale groupoid and let $\mathscr{H}$ be a family of open subgroupoids. We say that $G$ is \textit{topologically principal relative to} $\mathscr{H}$ if the set
 \begin{equation}\label{eqn:definitionX_H}
     X_\msc{H} := \{u \in G^{(0)}\! : \exists \Gamma \in \operatorname{Bis}(G) \text{ and } H \in \mathscr{H} \text{ s.t. } u \in \Gamma^{-1}\Gamma \text{ and } \Gamma G^u_u\Gamma^{-1} \subseteq H \}
 \end{equation} is dense in $G^{(0)}$. 
\end{defn}

 It will be convenient to have a reciprocal terminology; when $G$ is topologically principal relative to a family  $\mathscr{H}$ of open subgroupoids we say that $\mathscr{H}$ {\em constrains the isotropy} of $G$. If the family is a singleton $\mathscr{H} = \{H\}$, as will often be the case, we will simply say that $G$ is topologically principal relative to $H$, or that $H$ constrains the isotropy of $G$, and write $X_H := X_{\{H\}}$.

Before we state our main result, \thmref{thm:groupoidUniquenessThm} below, we give a few examples that illustrate the notion of relative topological principality, especially when framed within the various contexts discussed in \secref{sec:IdealProperty}.
\begin{example}
Let $G$ be a countable discrete group acting on itself  and consider the trivial subgroupoid $H =\{(e,e)\}$ of the transformation groupoid $G\ltimes G$.
Then $X_H = G$, hence  $G\ltimes G$ is topologically principal relative to $H$.
 \end{example}
 
 \begin{example}
 Let $G$ be an \'etale groupoid and let $H := G^{(0)}$ be the unit space, viewed as an open subgroupoid. 
If $\Gamma G^{u}_u \Gamma^{-1} \subseteq G^{(0)}$ for some open bisection $\Gamma$ then $\Gamma G^{u}_u \Gamma^{-1}$ is the trivial group, and so is $G^u_u$.  By \eqref{eqn:definitionX_H} 
 $u \in X_H$ if and only if  $ G^u_u = \{u\} $,  that is, $X_H$ is the set of units with trivial isotropy. As a consequence,  $G$ is topologically principal if and only if it is topologically principal relative to the unit space.
 \end{example}
 
 \begin{example}
 	Let $G$ be a second-countable Hausdorff \'etale groupoid. By \cite[Lemma~3.3]{brown2015cartan} the set $\{ u \in G^{(0)} \; | \; G^u_u \subseteq \operatorname{Iso}(G)^\circ \}$ is dense in $G^{(0)}$ hence $G$ is always topologically principal relative to $\operatorname{Iso}(G)^\circ$.
 \end{example}

  \begin{thm}\label{thm:groupoidUniquenessThm}
 Let $G$ be a  Hausdorff \'etale groupoid and suppose that $G$ is topologically principal relative to a family $\mathscr{H}$ of open subgroupoids . Then a representation $\rho$ of $C^*_r(G)$ is faithful if and only if its restriction to $C^*_r(H) \subseteq C^*_r(G)$ is faithful for each $H \in \mathscr{H}$.
 \end{thm}

To prove \thmref{thm:groupoidUniquenessThm} we will need to use isomorphisms of C*-algebras of subgroupoids that are inner in an appropriate sense.  It is natural that in a groupoid each element $\gamma \in G$ will act as a partial symmetry, but strictly speaking, one can only conjugate $\alpha \in G$ by $\gamma$ if $\ssrc{\gamma} = \ttrg{\alpha}$ and $\ssrc{\alpha} = \ttrg{\gamma^{-1}}$, in which case $\ssrc{\alpha} = \ttrg{\alpha}$, so $\alpha$ must be in the isotropy bundle of $G$. This will not be sufficient for our purposes so we broaden our understanding of what it means for $G$ to act on itself by conjugation by allowing an action of $\operatorname{Bis}(G)$ on the category Sub$(G)$ of subgroupoids of $G$. If $G$ is a group then $\operatorname{Bis}(G)$ may essentially be identified with $G$ and one may recover the inner automorphisms of $G$ from this action.
We formalize this notion in the following  definition.

\begin{defn}
By a \textit{partial automorphism} of an \'etale groupoid $G$ we mean an isomorphism $H \rightarrow K$ between two subgroupoids of $G$ in the category of  topological groupoids. We say that a partial automorphism $\phi$ is \textit{inner} if there is an open bisection $\Gamma \in \operatorname{Bis}(G)$ such that $\phi(h) = \Gamma h \Gamma^{-1}$ for all $h\in H$. Here $\Gamma h \Gamma^{-1} := (\textbf{s}|_\Gamma^{-1} \circ \textbf{t})(h) \cdot h \cdot (\textbf{t}|_{\Gamma^{-1}}^{-1} \circ \textbf{s})(h)$.
\end{defn}

   \begin{lem} 
 \label{ConjByGammaGroupoidIsoLem}
	Let $G$ be an \'etale groupoid, $\Gamma \in \operatorname{Bis}(G)$, and $H$ a subgroupoid of $G$ such that $H^{(0)} \subseteq \ssrc{\Gamma}$. Then the map $\operatorname{Ad}_\Gamma: H \rightarrow \Gamma H\Gamma^{-1} \text{ given by } \alpha \mapsto \Gamma \alpha \Gamma^{-1} := (\textbf{s}|_\Gamma^{-1} \circ \textbf{t})(\alpha) \cdot \alpha \cdot (\textbf{t}|_{\Gamma^{-1}}^{-1} \circ \textbf{s})(\alpha)$ is an inner partial automorphism of $G$.
 \end{lem}

 \begin{proof}
     Let $(\alpha, \beta) \in H^{(2)}$. Since $\Gamma$ is a bisection and $H^{(0)} \subseteq \ssrc{\Gamma}$ there are unique elements $\gamma_1, \gamma_2, \gamma_3 \in \Gamma$ such that 
     \[
     \ttrg{\alpha} = \ssrc{\gamma_1}, \quad    \ssrc{\alpha} = \ttrg{\gamma_2^{-1}} = \ssrc{\gamma_2} = \ttrg{\beta}, \quad \text{and } \ssrc{\beta} = \ttrg{\gamma_3^{-1}}.
     \]
     
 	Thus $\Gamma \alpha \Gamma^{-1} \Gamma \beta \Gamma^{-1} = \gamma_1 \alpha \gamma_2^{-1}\gamma_2 \beta \gamma_3^{-1} = \gamma_1 \alpha \beta \gamma_3^{-1} = \Gamma \alpha \beta \Gamma^{-1}$ so $\operatorname{Ad}_\Gamma$ is a homomorphism. Continuity follows from continuity of the source, range, and multiplication maps. Since $\operatorname{Ad}_{\Gamma^{-1}}$ is a continuous inverse, $\operatorname{Ad}_{\Gamma}$ is an isomorphism of topological groupoids. 
 \end{proof}

\begin{defn}
 Let $B$ and $C$ be C*-subalgebras of a C*-algebra $A$. We say a subalgebra isomorphism $\phi: B \rightarrow C$ of $A$ is {\em inner} if there is a partial isometry $v$ in the universal enveloping von Neumann algebra $A^{**}$ of $A$ such that $\phi(b) = vbv^*$ and $\phi^{-1}(c) = v^*cv$ for all $b \in B$ and $c \in C$. 
 \end{defn}
 
 \begin{remark}
 	We use the term ``subalgebra isomorphism''  to distinguish this notion from the more restrictive one of partial automorphism used in connection to partial actions of groups; specifically, partial automorphisms are subalgebra isomorphisms for which domain and range are 2-sided ideals. This feature is required only for the definition of a crossed product, and does not concern us here.
 \end{remark}
 
 It is a standard theme that inner automorphisms preserve more structure than outer ones. For our approach to detecting ideals, the most important consequence of innerness is the following.
  
 \begin{lem}\label{innerSubalgIsoLem}
 Let $B$ and $C$ be C*-subalgebras of a C*-algebra $A$. Suppose $\rho: A \rightarrow \mbb{B}(\mc{H})$ is a representation and $\phi: B \rightarrow C$ is an inner subalgebra isomorphism of $A$. Then $\rho$ is faithful on $B$ if and only if $\rho$ is faithful on $C$.
 \end{lem}
 \begin{proof}
 	Suppose $\rho$ is faithful on $B$ and let $\tilde{\rho}$ denote the unique weak-* continuous extension to $A^{**}$. Let $c \in C$ be arbitrary and let $b \in B$ satisfy $c = vbv^* = \phi(b)$, where $v \in A^{**}$ is the partial isometry implementing $\phi$. Note that $\norm{b} = \norm{c}$ because $\phi$ is an isomorphism. Since $\rho$ is a $*$-homomorphism $\norm{\rho(c)} \leq \norm{c}$. As for the reverse inequality we have
 	\begin{multline*}
	\norm{c} = \norm{b} = \norm{\rho(b)} = \norm{\rho(v^*vbv^*v)} = \norm{\tilde{\rho}(v^*)\rho(vbv^*)\tilde{\rho}(v)} \\ \leq \norm{\rho(vbv^*)} = \norm{\rho(c)};
\end{multline*}

 	thus $\rho$ is faithful on $C$.
	The converse follows by symmetry because $\phi\inv :C \rightarrow B$ is an inner subalgebra isomorphism.
 \end{proof}

 Next we adapt Lemma~3.7 and  Proposition~3.8 from \cite{kennedyIdealIntersectionProperty2021} and show that the  the reduced C*-algebra of a groupoid $G$ has an action of $\operatorname{Bis(G)}$  by subalgebra isomorphisms of $C^*_r(G)$.  
 \begin{lem}\label{InducedPArtialIsoActionLem}
 	Suppose $G$ is a Hausdorff \'etale groupoid, and let $S$ denote the collection of partial isometries in the enveloping von Neumann algebra $C^*_r(G)^{**}$. Then there is a multiplicative and involutive map $v: \operatorname{Bis}(G) \rightarrow S$ mapping $\Gamma \mapsto v_\Gamma$ such that for every $\Gamma \in \operatorname{Bis}(G)$ and every open subgroupoid $H$ such that $H^{(0)} \subseteq \ssrc{\Gamma}$,
 	\begin{enumerate}
 		\smallskip\item the net $(f)_{0 \leq f \leq \mbb{1}_\Gamma} \subseteq C_c(\Gamma)$ converges to $v_\Gamma$ in the weak-* topology as $f$ increases to $\mbb{1}_\Gamma$.
   
 		\smallskip\item $v_{\Gamma}^*v_\Gamma = \mbb{1}_{\Gamma^{-1}\Gamma}$ and $v_\Gamma v_\Gamma^* = \mbb{1}_{\Gamma\Gamma^{-1}}$.
   
 		\smallskip\item If $g \in C_c(V)$ for some $V \in \operatorname{Bis}(H)$ then $v_\Gamma g = g(\Gamma^{-1} \cdot) \in C_c(\Gamma V)$, and $gv_{\Gamma}^* = g(\cdot \Gamma) \in C_c(V\Gamma^{-1})$.
   
 		\smallskip\item The isomorphism $\phi: C^*_r(H) \rightarrow C^*_r(\Gamma H \Gamma^{-1})$ induced by the groupoid isomorphism $H \rightarrow \Gamma H \Gamma^{-1}$ is implemented by the conjugation $\operatorname{Ad}_{v_\Gamma}: a \mapsto v_\Gamma a v_\Gamma^* $ and is therefore an inner subalgebra isomorphism.
 	\end{enumerate}
 \end{lem}
 \begin{proof}
  Let $\Gamma \in \operatorname{Bis}(G)$. By the C*-identity there is an isometric isomorphism of Banach spaces $C_0(\Gamma) \xrightarrow{\cong} \overline{C_c(\Gamma)} \subseteq C^*_r(G)$ hence we have isometric embeddings $B^\infty(\Gamma) \hookrightarrow C_0(\Gamma)^{**} \hookrightarrow C^*_r(G)^{**}$ where $B^\infty(\Gamma)$ is the algebra of bounded Borel functions on $\Gamma$ and the embedding $B^\infty(\Gamma) \hookrightarrow C_0(\Gamma)^{**}$ is given by sending $f \in B^\infty(\Gamma)$ to the linear functional on $C_0(\Gamma)^* $ given by  $\mu \mapsto \int_\Gamma f d\mu$. Since the increasing net $(f)_{0 \leq f \leq \mbb{1}_{\Gamma}}$ in $C_c(\Gamma)$ converges pointwise to $\mbb{1}_\Gamma \in B^\infty(\Gamma)$,  the monotone convergence theorem implies 
  \[
  \Big(\int_\Gamma f d\mu\Big)_{0 \leq f \leq \mbb{1}_{\Gamma}} \longrightarrow \int_\Gamma \mbb{1}_{\Gamma} d\mu,
  \]
  which is to say $(f)_{0 \leq f \leq \mbb{1}_{\Gamma}} \rightarrow \mbb{1}_{\Gamma}$ in the weak-* topology on $B^\infty(\Gamma)$ and hence in $C^*_r(G)^{**}$ because $C_0(\Gamma)^{**}$ is weak-* closed in the latter. Setting $v_\Gamma := \mbb{1}_{\Gamma}$ gives part (1). Next let $g \in C_c(V)$ for an open bisection $V$ of $H$. Then for all $0 \leq f \leq \mbb{1}_{\Gamma}$ we have 
 \[ (f*g)(x) =\begin{cases} 
      f(y)g(y^{-1}x) & \text{if  $\exists  y \in \Gamma$  (necessarily unique) such that } y^{-1}x \in V, \\
      0 & \text{otherwise.} 
   \end{cases}
\] 
The convolution operator  $\psi_g: B^\infty(\Gamma) \rightarrow B^\infty(\Gamma V)$ defined by $\psi_g(f) =f*g $ is continuous with respect to the topology of pointwise convergence on $B^\infty(\Gamma)$ and $B^\infty(\Gamma V)$, hence 
 \[
 (v_\Gamma g)(x) = \lim_{0 \leq f \leq \mbb{1}_\Gamma} (f * g)(x) = g(\Gamma^{-1}x) 
 \]
and similarly
  \[
  (gv_{\Gamma^{-1}})(x) = \lim\limits_{0 \leq f \leq \mbb{1}_{\Gamma^{-1}}}(g * f)(x) = g(x\Gamma),
  \]
  proving part (3). 
 To see that the map $\operatorname{Bis}(\Gamma) \rightarrow S$,  $\Gamma \mapsto v_\Gamma$ is multiplicative and involutive, observe that 
  \[
  v_{\Gamma_1} v_{\Gamma_2} = \lim_{0 \leq f \leq \mbb{1}_{\Gamma_2}} v_{\Gamma_1} f = \lim_{0 \leq f \leq \mbb{1}_{\Gamma_2}}  f(\Gamma_1^{-1} \cdot ) = \lim_{0 \leq f \leq \mbb{1}_{\Gamma_1 \Gamma_2}}  f = v_{\Gamma_1 \Gamma_2} 
  \]
  and similarly $v_{\Gamma}^* = v_{\Gamma^{-1}}$. Moreover, this gives part (2) because
  
 $$v_\Gamma v_\Gamma^* = \lim_{0 \leq f \leq \mbb{1}_{\Gamma^{-1}}} v_\Gamma f = \lim_{0 \leq f \leq \mbb{1}_{\Gamma^{-1}}} f(\Gamma^{-1}\cdot ) = \mbb{1}_{\Gamma \Gamma^{-1}}$$
 and similarly $v_\Gamma^* v_\Gamma = \mbb{1}_{\Gamma^{-1}\Gamma}$.
 
 Finally we prove part (4). By (3) we have $v_\Gamma g v_\Gamma^* = g(\Gamma^{-1} \cdot \Gamma) \in C_c(\Gamma H \Gamma^{-1})$ for each $g \in C_c(H)$. By hypothesis we have $\Gamma^{-1}\Gamma H \Gamma^{-1}\Gamma = H$ and thus $\operatorname{Ad_{v_{\Gamma}^*}}:C_c(\Gamma H \Gamma^{-1}) \rightarrow C_c(H)$ is the inverse of $\operatorname{Ad}_{v_\Gamma}$. Since $\norm{a} = \norm{v_\Gamma^* v_\Gamma a v_\Gamma^* v_\Gamma} \leq \norm{v_\Gamma a v_\Gamma^{*}} \leq \norm{a}$, the result follows.
 \end{proof}
\begin{rem}
We point out that \lemref{InducedPArtialIsoActionLem} is about the reduced groupoid C*-algebra $C^*_r(G)$ instead of the universal one considered in \cite[Lemma~3.7.]{kennedyIdealIntersectionProperty2021} and  that we do not need to assume the unit space $G^{(0)}$ to be compact because we do not rely upon the notion of $G$-C*-algebra used in \cite{kennedyIdealIntersectionProperty2021}. One should also notice that the set of partial isometries in a von Neumann algebra does not typically form an inverse semigroup under multiplication as claimed in \cite[Lemma~3.7.]{kennedyIdealIntersectionProperty2021}. This does not affect the results because all that is required is that  the image of the map be an inverse semigroup.
 \end{rem}
 
The following two lemmas generalize results from \cite{brown2015cartan} to a collection of subalgebras and subgroupoids; we are indebted to Kwa\'{s}niewski
for pointing out that second countability is not needed for part (b) of \cite[Lemma 3.3]{brown2015cartan}.
 \begin{lem}[cf. \cite{brown2015cartan} Theorem 3.2]\label{StateExtensionFaithfullnessLem}
 	Let $\msc{B}$ be a collection of C*-subalgebras of a C*-algebra $A$. For each $B \in \msc{B}$ suppose that $S_B$ is a collection of states on $B$ such that each state $\varphi \in S_B$ has a unique extension  to a state $\tilde{\varphi}$ on $A$. Let $S := \bigcup_{B \in \msc{B}} S_B$ and suppose that the sum of GNS representations $\bigoplus_{\varphi \in S} \pi_{\tilde{\varphi}}$ of the extensions of the states in $S$ is faithful on $A$. Then a representation $\rho: A \rightarrow \mbb{B}(\mc{H})$ is faithful if and only if it is faithful on $B$ for each $B\in \msc{B}$.
 \end{lem}
 \begin{proof}
 	 Suppose that the restriction of $\rho: A \rightarrow \mbb{B}(\mc{H})$ to $B$ is faithful for each $B \in \msc{B}$. Set $J := \operatorname{Ker} \rho$, let $B \in \msc{B}$ be arbitrary, and observe that $J \cap B = \{0\}$. By e.g. \cite[Corollary~1.8.4.]{dixmierAlgebras1982}, $A_0 := J + B$ is a C*-subalgebra of $A$ with ideal $J$ and quotient $A_0/J \cong B$. Denote the quotient map by $\gamma: A_0 \rightarrow B$. Let $\varphi \in S_B$ be arbitrary. Then $\tilde{\varphi}(a) = \varphi(\gamma(a))$ for all $a \in A_0$, for otherwise any extension of $\varphi\circ \gamma$ to $A$ would extend $\varphi$ and be different from $\tilde \varphi$. 
   
    Let $x \in J$ and denote by $\xi_{\tilde{\varphi}}$ the GNS vector of $\tilde\varphi$. For every $a \in A$,
    \begin{multline*}
    \langle \pi_{\tilde{\varphi}}(x^*x) \pi_{\tilde{\varphi}}(a) \xi_{\tilde{\varphi}},
    \pi_{\tilde{\varphi}}(a) \xi_{\tilde{\varphi}} \rangle =
    \langle \pi_{\tilde{\varphi}}(a^*x^*xa)  \xi_{\tilde{\varphi}},
    \xi_{\tilde{\varphi}} \rangle \\ =
    \tilde{\varphi}(a^*x^*xa)  =
    \varphi(\gamma(a^*x^*xa)) =
    0
    \end{multline*}
      because $a^*x^*xa \in J$. Therefore $\pi_{\tilde{\varphi}}(x^*x)= 0$, and since $\bigoplus_{\varphi \in S} \pi_{\tilde{\varphi}}$ is faithful, and $B \in \msc{B}$ and $\varphi \in S_B$ are arbitrary we must have $x =0$, proving that $\rho$ is faithful.
 \end{proof}
 
 \begin{lem}[cf. \cite{brown2015cartan} Lemma 3.3]\label{bfbLem}
 Let $G$ be a Hausdorff \'etale groupoid that is topologically principal relative to a family $\msc{H}$ of open subgroupoids. Let $X_{\msc H}$ be the set defined in \eqref{eqn:definitionX_H} and suppose that $u\in X_\msc{H}$ and $\Gamma \in \operatorname{Bis}(G)$ satisfy $u \in \Gamma^{-1}\Gamma$ and $\Gamma G^u_u \Gamma^{-1} \subseteq H$ for some $H \in \msc{H}$. Then for each $f \in C_c(G)$ there exists $b \in C_c((\Gamma^{-1}H\Gamma)^{(0)})^{+}$ such that $\lVert b \rVert = b(u) = 1$ and $bfb\in C_c(\Gamma^{-1}H\Gamma)$.
 \end{lem}
 \begin{proof}
 	 Let $f \in C_c(G)$ and write $f = \sum_{D \in F} f_D$ where $F$ is a finite collection of precompact open bisections and $\operatorname{supp}(f_D)\subseteq D$. For each $D \in F$ choose an open neighborhood $V_D$ of $u$ in $G^{(0)}$ as follows.
 	\begin{itemize}
 		\item If there exists $\alpha \in D\cap G^u_u$ then $D\cap \Gamma^{-1}H\Gamma$ is an open neighborhood of $\alpha$, hence $V_D := \ssrc{D\cap \Gamma^{-1}H\Gamma} \cap \ttrg{D\cap \Gamma^{-1}H\Gamma}$ is an open neighborhood of $u$ and $V_D D V_D \subseteq D \cap \Gamma^{-1}H \Gamma$.
 		
 		\item If $D\cap G^u_u=\emptyset$ but there exists $\alpha \in D$ with $r(\alpha) = u \neq s(\alpha)$ or $r(\alpha) \neq u = s(\alpha)$; then take an open neighborhood $ U \subseteq D$ of $\alpha$ such that $s(U) \cap r(U) = \emptyset$ and set $V_D := r(U)$ so that $V_D D V_D = \emptyset$. 
 		\item If $u \not\in r(D) \cup s(D)$ then take an open neighborhood $V_D$ of $u$ such that $f_D|_{V_D D V_D} = 0$.
 	\end{itemize}
 	Let $V:= \cap_{D\in F} V_D$. Then $V$ is open and contains $u$. Choose $b\in C_c(V)^+$ such that $b(u) = \norm{b} = 1$. By construction $\operatorname{supp}(bfb) \subseteq \Gamma^{-1}H\Gamma$ as desired.
 \end{proof}

 \begin{lem} [cf. \cite{brown2015cartan} Lemma 3.5]\label{CompressibleModGammaHGammaLem}
 	 
    Under the same assumptions as in \lemref{bfbLem}, let $a \in C^*_r(G)$ and $\epsilon >0$. Then there exists $c \in C^*_r(\Gamma^{-1} H \Gamma)$ and $b \in C_c(G^{(0)})^+$ such that $\norm{b} = 1$ and $\varphi(b) = 1$ for all states $\varphi$ that factor through $C_r^*(G_u^u)$ and such that $\norm{bab-c}<\epsilon$.
 \end{lem}
\begin{proof}
	Let $\epsilon>0$ and $a\in C^*_r(G)$. Let $f \in C_c(G)$ such that $\norm{a - f} < \epsilon$. Apply \lemref{bfbLem} to find $b \in C_c(G^{(0)})^+$ such that $\norm{b} = b(u)=1$ and $bfb \in C_c(\Gamma^{-1} H \Gamma)$. Let $\vartheta_u: C^*_r(\Gamma^{-1} H \Gamma) \rightarrow C^*_r(G^u_u)$ be the completely positive contraction that extends the restriction map $C_c(\Gamma^{-1} H \Gamma) \rightarrow C_c(G^u_u), \ f \mapsto f|_{G^u_u}$ by \cite[Lemma~1.2]{christensenNonExoticCompletions2022}. Then $\vartheta_u(b) = \delta_u$ is the identity of $C^*_r(G^u_u)$ hence if $\varphi$ is a state on $C^*_r(\Gamma^{-1} H \Gamma)$ such that $\varphi = \psi \circ \vartheta_u$ for some state $\psi$ on $C^*_r(G^u_u)$ we have $\varphi(b) = \psi(\delta_u) = 1$. Moreover if we set $c = bfb$ we have $$\norm{bab- c} = \norm{bab- bfb} < \epsilon$$ as desired.
\end{proof}

Before we prove our main theorem we need one final lemma.

\begin{lem}\label{RestrictionsCommuteLemma}
Suppose that $G$ is a Hausdorff \'etale groupoid and let $x \in G^{(0)}$. Let 
\begin{itemize}
    \item $\Phi: C^*_r(G) \rightarrow C_0(G^{(0)})$ be the canonical faithful conditional expectation, 
     \item $\vartheta_x: C^*_r(G) \rightarrow C^*_r(G^x_x)$ be the completely positive contraction extending the restriction map $C_c(G) \rightarrow C_c(G^{x}_x)$ as in \cite[Lemma~1.2.]{christensenNonExoticCompletions2022}, 
      \item $\tau: C^*_r(G^x_x) \rightarrow \C$ be the canonical faithful trace given by $\tau(a) = \langle \delta_x, a \delta_x\rangle$ and  \item $\operatorname{ev}_x: C_0(G^{(0)}) \rightarrow \C$ be evaluation at $x$.
\end{itemize}
    Then the following diagram commutes.
\[\begin{tikzcd}
	{C^*_r(G)} & {C_0(G^{(0)})} \\
	{C^*_r(G^x_x)} & {\mathbb{C}}
	\arrow["\tau"', from=2-1, to=2-2]
	\arrow["{\vartheta_x}"', from=1-1, to=2-1]
	\arrow["\Phi", from=1-1, to=1-2]
	\arrow["{\operatorname{ev}_x}", from=1-2, to=2-2]
\end{tikzcd}\]
\end{lem}
\begin{proof}
	If $f \in C_c(G)$ then $\tau(\vartheta_x(f)) = \langle \delta_x, f|_{G^x_x} \delta_x \rangle = f(x) = \operatorname{ev}_x \circ \Phi$. The result follows by continuity.
\end{proof}

We are now ready for the proof of our main result.  The approach we take is modelled on the proof of Theorem 3.1 in \cite{brown2015cartan}.
\begin{proof}[Proof of \thmref{thm:groupoidUniquenessThm}]
	For each $u \in X_\msc{H}$, $H \in \msc{H}$, and $\Gamma \in \operatorname{Bis}(G)$ such that $u \in \Gamma^{-1}\Gamma$ and $\Gamma G_u^u \Gamma^{-1} \subseteq H$ let $S_{u,\Gamma,H}$ denote the collection of states on $C^*_r(\Gamma^{-1} H \Gamma)$ that factor through $C^*_r(G^u_u)$. Let $\rho: C^*_r(G) \rightarrow \mbb{B}(\mc{H})$ be a representation that is faithful on the family of C*-subalgebras $\{ C^*_r(H) \; | \; H \in \msc{H} \}$. Then $\rho$ is faithful on $C^*_r(K) \subseteq C^*_r(H)$ for every open subgroupoid $K$ of $H$. Let $K := \ttrg{\Gamma} H \ttrg{\Gamma}$. Then $K \subseteq H$ is an open subgroupoid, with $K^{(0)} \subseteq \ssrc{\Gamma}$, and moreover $\Gamma^{-1}K\Gamma = \Gamma^{-1}H \Gamma$. By \lemref{InducedPArtialIsoActionLem} $\operatorname{Ad}_{v_{\Gamma}^*}:C^*_r(K) \rightarrow C^*_r(\Gamma^{-1}H\Gamma)$ is an inner subalgebra isomorphism. Hence by \lemref{innerSubalgIsoLem} faithfulness of $\rho$ on $C^*_r(H)$ implies faithfulness on $C^*_r(\Gamma^{-1}H\Gamma)$.

	Thanks to \lemref{CompressibleModGammaHGammaLem} we may apply \cite[Theorem~5.4.]{ArmstrongUniqueness}, and conclude that for each $u \in X_\msc{H}$, $H \in \msc{H}$ and $\Gamma$ with $H^{(0)} \subseteq \ssrc{\Gamma}$ such that $G^u_u \subseteq \Gamma^{-1}H\Gamma$, every state $\varphi \in S_{u,\Gamma,H}$ has a unique state extension $\tilde\varphi$ to $C^*_r(G)$.
 By  \lemref{StateExtensionFaithfullnessLem}, in order to prove the theorem it suffices to  prove that the direct sum $\pi_{S_{u,\Gamma,H}} := \bigoplus_{\varphi \in S_{u,\Gamma,H}} \pi_{\tilde{\varphi}}$ of the GNS representations of the extended states is faithful on $C^*_r(G)$. 
 
 In order to prove this, suppose $\pi_{S_{u,\Gamma,H}}(a)=0$ and consider the canonical faithful conditional expectation $\Phi_G: C^*_r(G) \rightarrow C_0(G^{(0)})$. 
 
 We claim that  $\Phi_G(a^*a) = 0$, which then forces $a=0$ because $\Phi_G$ is faitfhul. Suppose for contradiction that $\Phi_G(a^*a) \neq 0$. Then $\Phi_G(a^*a)$ is a nontrivial positive function on $G^{(0)}$ so by density of $X_\msc{H}$ in $G^{(0)}$ there is a $u \in X_\msc{H}$ such that $\Phi_G(a^*a)(u) > 0$. Let $\epsilon = \frac{\Phi_G(a^*a)(u)}{2}$. By  \lemref{CompressibleModGammaHGammaLem} we can find $b \in C_c(G^{(0)})^+$ with $\|b\| =1$ and $c \in C^*_r(\Gamma^{-1}H\Gamma)$ such that $\norm{ba^*ab - c} < \epsilon/2$ and $\varphi(b)=1$ for all states that factor through $C^*_r(G^u_u)$. 
 By assumption, $\pi_{S_{u,\Gamma,H}}(a) = 0$ and hence also $\pi_{S_{u,\Gamma,H}}(ba^*ab) = 0$ because $\norm{\pi_{S_{u,\Gamma,H}}(ba^*ab)} \leq \norm{\pi_{S_{u,\Gamma,H}}(a^*a)}=0$. If $\varphi \in S_{u,\Gamma,H}$ and $\{e_\lambda\}$ is an approximate unit in $C^*_r(G)$, then
\[
\tilde{\varphi}(ba^*ab) = \lim_\lambda \langle \pi_{\tilde{\varphi}}(ba^*ab)(e_\lambda + N_{\tilde{\varphi}}), (e_\lambda + N_{\tilde{\varphi}}) \rangle = 0
\]
	and therefore 
 \[
 |\tilde{\varphi}(c)| = |\tilde{\varphi}(ba^*ab - c)| \leq \norm{ba^*ab - c} < \epsilon/2.
 \]
	
	Let $\vartheta_u: C^*_r(G) \rightarrow C^*_r(G^u_u)$ be the completely positive contraction extending the restriction map. Then 
 \begin{align*}
 \norm{\vartheta_u(c)} &= \sup  \big\{|\psi(\vartheta_u(c))| \; | \; \psi \text{ is a state on }C^*_r(G^u_u)\big\} \\
 &\leq \sup \big\{ |\varphi(c)| \; | \; \varphi \in S_{u,\Gamma,H} \big\} \\
 &\leq \epsilon/2.
  \end{align*}

Now let $\Phi_{\Gamma,H}: C^*_r(\Gamma^{-1}H\Gamma) \rightarrow C_0((\Gamma^{-1}H\Gamma)^0)$ be the conditional expectation given by restricting functions to the unit space. For $x\in (\Gamma^{-1}H\Gamma)^0$ let $\operatorname{ev}_x$ be  evaluation at $x$ on $C_0((\Gamma^{-1}H\Gamma)^0)$. Then $\operatorname{ev}_x \circ \Phi_{\Gamma,H}$ is a state on $C^*_r(\Gamma^{-1}H\Gamma)$; moreover for each $x \in X_{\Gamma^{-1}H\Gamma} = \{x \in (\Gamma^{-1}H\Gamma)^0 \; | \; G_x^x \subseteq \Gamma^{-1}H\Gamma \}$ the state $\operatorname{ev}_x \circ \Phi_{\Gamma,H}$ factors through $C^*_r(G^x_x)$ by \lemref{RestrictionsCommuteLemma}. Let $\tau:C^*_r(G^u_u) \rightarrow \C$ be the canonical tracial state. Since $\Phi_G|_{\Gamma^{-1}H\Gamma} = \Phi_{\Gamma,H}$ we have
$$|\operatorname{ev}_u \circ \Phi_G(c)| = |\operatorname{ev}_u \circ \Phi_{\Gamma,H}(c)| = |\tau(\vartheta_u(c))| \leq \norm{\vartheta_u(c)} \leq \epsilon/2.$$
Since $\Phi_G$ is $C_0(G^{(0)})$-bilinear and $\operatorname{ev}_u$ is multiplicative,
\[
\operatorname{ev}_u \circ \Phi_G(ba^*ab) = \operatorname{ev}_u(b) (\operatorname{ev}_u \circ \Phi_G)(a^*a) \operatorname{ev}_u(b) = (\operatorname{ev}_u \circ \Phi_G)(a^*a).
\]
Thus 
\begin{align*}
	|\operatorname{ev}_u \circ \Phi_G(a^*a)| &\leq |\operatorname{ev}_u \circ \Phi_G(a^*a) - \operatorname{ev}_u \circ \Phi_G(c) + \operatorname{ev}_u \circ \Phi_G(c)|\\
	&\leq |\operatorname{ev}_u \circ \Phi_G(ba^*ab) - \operatorname{ev}_u \circ \Phi_G(c)| + |\operatorname{ev}_u \circ \Phi_G(c)|\\
	&\leq \norm{ba^*ab - c} + |\operatorname{ev}_u \circ \Phi_G(c)|\\
	&< \epsilon.
\end{align*}
But this contradicts our choice of $\epsilon$, hence $\Phi_G(a^*a) = 0$ and therefore $a=0$. It follows that $\pi_{S_{u,\Gamma,H}}$ is injective.
\end{proof}
  {\begin{rem}
     Before we launch into  applications of the main result in the following two sections, we would like to point out that, as indicated after the definition, it is often the case that a singleton $\msc{H}$ suffices to constrain  the isotropy in a useful way. Of course, in a trivial way this is always true because it is possible to choose $\msc{H} =\{ \msc{G}\}$. The point is that by making the open subgroupoids as small as possible, maybe at the cost of taking more of them, one gets a useful criterion. As an example, consider a direct sum of full matrix algebras, which is the groupoid C*-algebra of a disjoint union of groupoids, in which a constraining family is obtained by including a trivial groupoid (consisting of a single point) in $\msc{H}$ for each summand.
 \end{rem}}

\section{ C*-algebras of Small Categories}\label{sec:smallcats}

In \cite{SpielbergSmallCats} Spielberg associates certain C*-algebras to left cancellative small categories (LCSCs), generalizing his earlier work on C*-algebras of categories of paths \cite{SpielbergPathCats}.  These include C*-algebras of $k$-graphs and quasi-lattice ordered groups.

\begin{defn}
	A \textit{left cancellative small category (LCSC)} is a small category $\mc{C}$ such that $xy = xz \implies y=z$ for all $x,y,z \in \mc{C}$ with $xy, xz \in \mc{C}$. 
\end{defn}

 A left cancellative monoid is just a LCSC with one object, and the aim of this section is to use the results of the previous section to provide a version of the faithfulness/uniqueness theorem of Laca and Sehnem \cite[Theorem~5.1]{LS} in the more general setting of LCSCs.  
 
 The key fact is that left cancellation implies that the partially defined left multiplication maps given by 
 $\ssrc{c}\mc{C} \rightarrow c \mc{C}, \ x \mapsto cx$ for $c \in \mc{C}$  are partial bijections. 
 Hence for each $c \in \mc{C}$ there is a partial isometry $T_c: \ell^2(\mc{C}) \rightarrow \ell^2(\mc{C})$ determined by 
	\[ T_c \delta_x = \begin{cases} 
     \delta_{cx} & \text{ if } \textbf{s}(c)=\textbf{t}(x) \\
      0 & \text{otherwise}. 
   \end{cases}
\]
	The map $c\mapsto T_c$ is the left regular representation $T: \mc{C} \to \mbb{B}(\ell^2(\mc{C}))$, and we are interested in the C*-algebra generated by its image.

\begin{defn}
	The left reduced C*-algebra of the LCSC $\mc{C}$ is  
	\[
 C^*_r(\mc{C}) := C^*(\{ T_c \; | \; c \in \mc{C}  \}) 
 \subseteq \mbb{B}(\ell^2(\mc{C})).
 \]
 \end{defn}

On our way  towards Spielberg's groupoid model for $C^*_r(\mc{C})$ we will use the inverse semigroup approach to $C^*_r(\mc{C})$ described in  {\cite{PO2020}, see also \cite{LiGarsideCats}}, because it allows us to make easy comparison with semigroup C*-algebras.  Thus we turn our attention to a natural inverse semigroup arising from a LCSC.

\begin{defn}(\cite[Definition~2.3]{LiGarsideCats})
	The {\em left inverse hull} $\mc{I}_l(\mc{C})$ of a LCSC $\mc{C}$ is the inverse subsemigroup of the symmetric monoid $\mc{I}(\mc{C})$ generated by the left multiplication maps $\ssrc{c}\mc{C} \rightarrow c \mc{C}, \ x \mapsto cx$ for $c \in \mc{C}$.  
\end{defn}
When no confusion can arise we may omit $\mathcal{C}$ from the notation and simply write $\mathcal{I}_l$ instead of $\mc{I}_l(\mc{C})$.
 We often abuse notation and identify an element $c \in \mc{C}$ with the corresponding left multiplication map $c \in \mc{I}_l(\mc{C})$. We denote its inverse map in $\mc{I}_l$ by $c\inv$; by definition,  this is the map defined by $c \mc{C} \rightarrow \ssrc{c}\mc{C}, \ cx \mapsto x$, but when $c$ is invertible in $\mc{C}$, the inverse map $c\inv$ coincides with (left multiplication by) the element $c\inv \in \mc{C}$.
Notice that the set $\mc{C}^* := \mc{C} \cap \mc{C}\inv$ of invertibles in $\mc{C}$ is a groupoid.

An arbitrary element $s\in \mc{I}_l$ may be written in the form $d_n\inv c_n\cdots d_1\inv c_1$ where  $d_i,c_i \in \mc{C}$ for $i = 1, 2, \ldots , n$ and the successive concatenations are pairwise composable in $\mc{C}$; these are the {\em zigzag maps} of \cite[Definition 2.6]{SpielbergSmallCats}.  
 The representation $T :\mc{C} \to \mbb{B}(\ell^2(\mc{C}))$ can be extended to all of $\mc{I}_l$ by setting $T_{d_n^{-1}c_n\cdots d_1^{-1}c_1} := T_{d_{n}}^*T_{c_n}\cdots T_{d_{1}}^*T_{c_1}$,
 see \cite[Lemma 11.1]{SpielbergSmallCats}.
 Clearly  $C^*(\{ T_c \; | \; c \in \mc{C}  \}) = \overline{\operatorname{span}} \{ T_{s} \; | \; s \in \mc{I}_l \}$.
So $C^*_r(\mc{C})$ is precisely the C*-algebra denoted by $\mathcal T_{\ell}(\mc{C})$ and called the regular Toeplitz algebra of $\mc{C}$ by Spielberg in \cite[Definition 11.2]{SpielbergSmallCats}.

	As pointed out at the beginning of \cite[Section 8]{SpielbergSmallCats}, there are C*-algebraic relations satisfied in $C_r^*(\mc{C}) = \mathcal T_{\ell}(\mc{C})$ that are not captured by the bare inverse semigroup structure of $\mc{I}_l(\mc{C})$, see also \cite[Section 3]{LS}. Indeed, in some cases the equality $f = \vee_{i=1}^n e_i$ holds for idempotents in $\mc{I}_l$ and hence is respected by the left regular representation, but it is not derived from the inverse semigroup structure alone. One needs to consider  the order theoretic relations between idempotents. 
 Recall that $E(\mc{I}(X))$ denotes the semilattice of idempotents of the symmetric monoid $\mc{I}(X)$ on a set $X$, and that $E(\mc{I}(X))$ is canonically isomorphic to the power set $\mc{P}(X)$, which is a lattice. We use this isomorphism to identify idempotents in $\mc{I}_l(\mc{C})$ with subsets of $\mc{C}$.

\begin{rem}
The zero element in the symmetric monoid $\mc{I}(\mc{C})$ is the partial bijection whose domain and range are the empty set; it is also the zero of the semilattice of idempotents  $E(\mc{I}(\mc{C})) \cong \mathcal{P(C)}$. We denote this element by $0$. Whether $0 \in \mc{I}_l$ depends on whether the empty set may be written as an intersection $\bigcap_{s \in F} \dom{s}$ for a finite set of elements $F\subseteq \mc{I}_l$.
\end{rem}

The images of the idempotents in $\mc{I}_l$ form a commuting family of projections that is closed under multiplication, so their closed linear span is a commutative C*-subalgebra of $C^*_r(\mc{C})$.

\begin{defn}
	Let $\mc{C}$ be a LCSC.  We define the {\em diagonal subalgebra} of $C^*_r(\mc{C})$ to be 
	$D_r(\mc{C}):= \overline{\operatorname{span}}\{T_e \; | \; e \in E(\mc{I}_l)\}.$
\end{defn}

 We note that when $0$ is in $\mc{I}_l$ the element $T_0$ is the zero element of the C*-algebra.


 \begin{defn} 
    Let $E$ be a (meet) semilattice viewed as a semigroup of idempotents and let $\{ 0,1\}$ be the two element boolean algebra. A \textit{semicharacter} $\chi$ is a nonzero multiplicative map $\chi: E \rightarrow \{0,1\}$. The space $\widehat{E}$ is the collection of semicharacters on $E$ with the topology of pointwise convergence.
\end{defn} 

In order to lighten the notation, for the remainder of this section we make the following simplifications: 
\begin{itemize}
\item $\mc{C}$ is a LCSC
\item $\mc{I}_l = \mc{I}_l(\mc{C})$ is the left inverse hull of $\mc{C}$
\item $E = E(\mc{I}_l)$ is the semilattice of idempotents of $\mc{I}_l$
\item $D_r = D_r(\mc{C})$ is the diagonal subalgebra of $C^*_r(\mc{C})$
\item $\Omega := \{\chi \in \widehat{E} \; | 
	\; \chi(f) = \vee_{i=1}^n \chi(e_i), \text{ whenever } f, e_i  \in \mc{I}_l \text{ and } f = \vee_{i=1}^n e_i\}$ is the subspace of $\widehat{E}$ consisting of semicharacters preserving any finite suprema that exist in $\mc{I}_l$.
\end{itemize}

Next we show that the necessary extra relations on $E$ can be captured via the spectrum of $D_r$.
\begin{lem}\label{SpecDiagLem}
	The map that sends a character $\chi \in \operatorname{Spec}(D_r)$
 to its restriction $\chi|_E$ to the subset $ \{ T_e \; | \; e \in E\} \subset D_r$ 
 is a homeomorphism of $\operatorname{Spec}(D_r)$ onto $\Omega$.
\end{lem}
\begin{proof}
	The map $T: E \rightarrow D_r, \ e \mapsto T_e$ is a faithful representation of the semigroup $E$. The restriction $\chi|_E$  of $\chi \in \operatorname{Spec}(D_r)$ to this semigroup is a semicharacter, but it also preserves joins in $E$ because whenever $e,f \in E$ with $e \vee f \in E$ we have $T_{e \vee f} = T_e + T_f - T_{ef}$, which is respected by the multiplicative linear functional $\chi$. Conversely every $\chi \in \Omega$ preserves the multiplicative and linear relations among the $T_e$ hence extends to a character on $D_r = \overline{\operatorname{span}}\{T_e \; | \; e \in E\}$. This shows that restriction is a bijection between $\operatorname{Spec}(D_r)$ and $\Omega$. To show that this bijection is indeed a homeomorphism we note that $D_r$ admits a surjection $C^*(E) \rightarrow D_r$ from the universal C*-algebra generated by a representation of the inverse semigroup $E$. Explicitly $C^*(E)$ is generated by projections $\{t_e \; | \; e \in E\}$ subject to the relations $t_et_f = t_{ef}$ and $t_{e^*} = t_e^*$. This realizes $\operatorname{Spec}(D_r)$ as a subspace of $\operatorname{Spec}(C^*(E))$. On the other hand the universal property of $C^*(E)$ gives a 1-1 correspondence between semicharacters $E \rightarrow \{ 0,1\}$ and characters $C^*(E) \rightarrow \C$. Since the weak-* topology on $\operatorname{Spec}(C^*(E))$ is the topology of pointwise convergence on $\widehat{E}$, and since the diagram 
\[\begin{tikzcd}
	{\operatorname{Spec(C^*(E))}} & {\widehat{E}} \\
	{\operatorname{Spec(D_r)}} & \Omega
	\arrow["{\chi \mapsto \chi|_E}", from=1-1, to=1-2]
	\arrow["{\chi \mapsto \chi|_E}", from=2-1, to=2-2]
	\arrow[hook, from=2-1, to=1-1]
	\arrow[hook, from=2-2, to=1-2]
\end{tikzcd}
\]
is commutative, the result follows because the bottom horizontal map can be seen as  a restriction of the top horizontal map, which is a homeomorphism.
\end{proof}

\begin{rem}
	An alternate proof of \lemref{SpecDiagLem} may be obtained using \cite[Corollary~5.6.26.]{cuntzKTheoryGroupAlgebras2017}.
\end{rem}
Just as in the case of  semigroups, here too the spectrum of the diagonal can be seen as a compactification of $\mc{C}$ via a canonical 
 dense embedding $\mc{C} \hookrightarrow \Omega$ as {\em principal characters}.
\begin{lem}\label{semicharactersDenseLem}\cite[Lemma~2.12.]{LiGarsideCats}
For each $c \in \mc{C}$ define $ \chi_c:E \rightarrow \{0,1\} $ by
 \[
 \chi_c(e) = \begin{cases}   1 & \text{ if } cc^{-1} \leq e \\ 0&\text{ otherwise.} 
 \end{cases} 
 \]
 Then $\chi_c$ is a semicharacter and 
$\{\chi_c \; | \; c \in \mc{C}\}$ is a dense subset of $\Omega$.	
\end{lem}
\begin{proof}
	One easily verifies that each $\chi_c \in \Omega$. Recall that a basic open subset of the spectrum $\widehat{E}$ of the semilattice $E$ is of the form $\widehat{E}(e,F) := \{ \chi \; | \; \chi(e) = 1, \chi(f) = 0 \text{ for all }f \in F\}$. To see density, first note that $E$ is a semilattice which is concretely represented as a subset of the power set of $\mc{C}$, thus we can view elements of $E$ as sets. Density then follows immediately from the observation that if $\widehat{E}(e,F) \neq \emptyset$ then $e \setminus \cup_{f \in F} f \neq \emptyset$, in which case $\chi_c \in \widehat{E}(e,F)$ for each $c \in e \setminus \cup_{f \in F} f$.
\end{proof}

Every inverse semigroup $S$ has a canonical action by partial homeomorphisms of the spectrum $\widehat{E(S)}$ of its semilattice of idempotents. The domain of $s \in S$ is $D_{s^{-1}s} = \{ \chi\in \widehat{E(S)} \; | \; \chi(s^{-1}s) = 1 \}$ and  the image under $s$ of  $\chi \in D_{s\inv s}$ is the character $s\cdot \chi$,  given by $(s \cdot \chi) (e) = \chi(s\inv e s)$. It is easy to verify that $s\cdot \chi \in D_{ss\inv}$.  We next show that $\Omega$ is invariant under this action.

\begin{lem}\label{lem:Omegaisinvariant}
	$\Omega$ is an $\mc{I}_l$-invariant subset of $\widehat{E}$.
\end{lem}
\begin{proof}
	Let $\chi \in \Omega$ and suppose $\chi \in D_{s^{-1}s}$ for some $s \in \mc{I}_l$. We claim that $s\cdot \chi \in \Omega$. Suppose $f$ is an idempotent which can be written as a join $f = \bigvee_{i=1}^n e_i$ of a finite set of idempotents. Then 
	
\begin{multline*}(s\cdot\chi)\left(\bigvee_{i=1}^n e_i\right) = \chi\left(s^{-1} \bigvee_{i=1}^n e_i s\right) = \chi\left(\bigvee_{i=1}^n s^{-1} e_i s\right) \\
= \bigvee_{i=1}^n \chi(s^{-1}e_i s) = \bigvee_{i=1}^n (s\cdot \chi)(e_i)
\end{multline*}
	so $\chi \in \Omega$ as desired.  Note that the uses of distributivity in this calculation are justified by \cite[Proposition~1.2.2]{lawsonInverseSemigroups1998} and the fact that the above joins exist in $E$.
\end{proof}

In order to obtain a groupoid that includes the extra relations encoded in the diagonal, we simply  enlarge the underlying inverse semigroup in such a way that the idempotents become a Boolean algebra. 
\begin{defn}(cf. \cite[Definition~2.8.]{SpielbergSmallCats})
	Let $\mc{C}$ be a LCSC and let $E$ be the semilattice of idempotents of $\mc{I}_l$. The \textit{boolean closure} $\overline{E}$ of $E$ in $\mc{I(C)}$ is the collection of subsets of $\mc{C}$ of the form $\bigvee_{i=1}^n \left (e_i \setminus \bigvee_{j=1}^{k_i} f_{ij} \right)$ where $f_{ij}, e_i \in E$ and $f_{ij} \leq e_i$ for all $1 \leq j \leq k_i$ and $1 \leq i \leq n$. 
 We also enlarge the left inverse hull $\mc{I}_l(\mc{C})$ to 
	\[
 \overline{\mc{I}}_l (\mc{C}):= \{se \; | \; s \in \mc{I}_l, e \in \overline{E}, e \leq s^{-1}s \}
 \]
 which is an inverse semigroup with semilattice of idempotents $\overline{E}$. 
\end{defn}
It follows from \lemref{lem:Omegaisinvariant} that the action of $\mc{I}_l$ on $\widehat{E}$ restricts to an action of $\mc{I}_l$ on $\Omega$. Now for every $\chi \in \Omega$ we may extend $\chi$ to an element $\overline{\chi}$ of $\widehat{\overline{E}}$ by setting $\overline{\chi}(\bigvee_{i=1}^n (e_i \setminus \bigvee_{j=1}^{k_i} f_{ij})) = \bigvee_{i=1}^n (\chi(e_i) \setminus \bigvee_{j=1}^{k_i} \chi(f_{ij}))$, where the operations on the right side are to be understood in the sense of the two-element Boolean algebra $\{0, 1\}$. This gives an embedding $\Omega \hookrightarrow \widehat{\overline{E}}$ and the canonical action of $\overline{\mc{I}_l}$ on $\widehat{\overline{E}}$ restricts to an action on $\Omega$.

\begin{defn}\label{LCSCGroupoidDef}( {\cite[p. 140]{Paterson1999}, see also \cite[Definition 4.6]{Exel2008}})
Let $\overline{\mc{I}}_l * \Omega = \{(s,\chi) \in \overline{\mc{I}}_l \times \Omega \; | \; \chi(s^{-1}s) = 1\}$, and define an equivalence relation on $\overline{\mc{I}}_l * \Omega$ by setting 
\[
(s,\chi) \sim (t, \psi)\iff \begin{cases} 
  \chi = \psi \text{ and}\\ se = te 
  \text{ for some idempotent } e  \in \overline{\mc{I}}_l \text{ with }  \chi(e) =1  .
\end{cases}
\]
We denote the equivalence class of a pair $(s,\chi)$ by $[s,\chi]$.

The \emph{transformation groupoid} of a left cancellative small category $\mc{C}$ has underlying set $\overline{\mc{I}}_l \ltimes \Omega := (\overline{\mc{I}}_l * \Omega ) /\sim $.  The source and target maps are given by $\ssrc{[r,\chi]} = \chi$ and $\ttrg{[r,\psi]} = r\cdot \psi$ respectively. The inverse of $[s,\chi] \in \overline{\mc{I}}_l \ltimes \Omega$ is given by $[s^{-1}, s \cdot \chi]$ and the product of $[s,t\cdot \chi]$ and $[t,\chi]$ is $[st,\chi]$. Then $\overline{\mc{I}}_l \ltimes \Omega$ becomes an \'etale groupoid when given the topology with basis of open sets $\{ [s, U] \; | \; s \in \overline{\mc{I}}_l, U \text{ is an open subset of }D_{s^{-1}s} \}.$
\end{defn}

In order to apply our relative topological principality results to the reduced C*-algebra of a LCSC we will rely on its realization as a groupoid C*-algebra given by Spielberg.

\begin{thm}\cite[Theorem~11.6.]{SpielbergSmallCats}
	Let $\mc{C}$ be a LCSC and suppose that the groupoid $\overline{\mc{I}}_l \ltimes \Omega$ is Hausdorff. Then $ C^*_r(\mc{C}) \cong C^*_r(\overline{\mc{I}}_l \ltimes \Omega) $.
\end{thm}

For the remainder of the section we let $\msc{G} := \overline{\mc{I}_l} \ltimes \Omega$.
We  need to understand the  isotropy in $\msc{G}$ next.  Recall that $\mc{C}^* =\mc{C} \cap \mc{C}\inv$ denotes the groupoid of invertible elements in $\mc{C}$ and note that $\msc{G}^{(0)} = (\overline{\mc{I}}_l \ltimes \Omega)^{(0)} = \Omega$.  Our description of isotropy will occur in \lemref{lem:categoryIsotropy} below; we first need a few technical lemmas.
\begin{lem}\label{firstEqualityLem}
Suppose that $e \in \mc{C}^{(0)}$ and $u \in \mc{C}^*$ satisfy $u^{-1}u = e = uu^{-1}$. Then $[u,\chi_e] \in \msc{G}^{\chi_e}_{\chi_e}$.
\end{lem}
\begin{proof}
	Suppose that $f \in E(\mc{I}_l)$. Then 
	\[ \chi_e(f) = \begin{cases} 
     1 & \text{ if } e \leq f \\
      0 & \text{otherwise,} 
   \end{cases}
\]
so we need to show that $e \leq f \iff e \leq u^{-1}fu$. If $e \leq f$ i.e. $ef = e$ then $uu^{-1}f = uu^{-1} \implies u^{-1}uu^{-1}fu = u^{-1}uu^{-1}u = u^{-1}u = e$ that is $e (u^{-1}f u) = e$. Conversely if $e (u^{-1}f u) = e$ then $(u^{-1}u)(u^{-1}fu) = uu^{-1}$ so conjugating by $u$ we get $(uu^{-1})(uu^{-1})f(uu^{-1}) = uu^{-1}$ which is equivalent to $ef = e$ since $e = uu^{-1} = u^{-1}u$ and idempotents commute.
\end{proof}

\begin{lem}\label{lem:preSecondEqualityLem}
If $s \in \mc{I}_l$, $p,q \in \mc{C}$, and $p \in \dom{s}$, then $s(pq) = s(p)q$.
\end{lem}
\begin{proof}
	If $s=d^{-1}c$ for $d,c \in \mc{C}$ then $cp = dr$ for some $r \in \mc{C}$ hence $cpq = drq$ and it follows that $s(pq) = rq = s(p)q$. We proceed by induction on the word length of $s=d_k^{-1}c_k \cdots d_1^{-1}c_1$. The desired result holds for $k=1$; assume it holds for all $k \leq \ell$. Let $p,pq \in \dom{s}$ where $s = d_{\ell + 1}^{-1}c_{\ell + 1} \cdots d_1^{-1}c_1$. Then $d_1^{-1}c_1(p)$ and $d_1^{-1}c_1(pq) = d_1^{-1}c_1(p)q$ are in $\dom{d_{\ell + 1}^{-1}c_{\ell + 1} \cdots d_2^{-1}c_2}$. By induction we have $d_{\ell + 1}^{-1}c_{\ell + 1} \cdots d_1^{-1}c_1(pq) = d_{\ell + 1}^{-1}c_{\ell + 1} \cdots d_2^{-1}c_2(d_1^{-1}c_1(p)q) = d_{\ell + 1}^{-1}c_{\ell + 1} \cdots d_1^{-1}c_1(p)q$.
\end{proof}

\begin{lem}\label{secondEqualityLem}
	Suppose that $e \in \mc{C}^{(0)}$ and $s \in \mc{I}_l$  satisfy $s^{-1}s = e = ss^{-1}$. Then $s \in \mc{C}^*$.
\end{lem}
\begin{proof}
	Let $p \in e\mc{C}e$. Then, using \lemref{lem:preSecondEqualityLem}, we have $s(p) = s(ep) = s(e)p$, so $s$ is left multiplication by $s(e)$.  Similarly $s^{-1}$ is left multiplication by $s^{-1}(e)$; hence $s^{-1}(e) = s(e)^{-1}$ and $s \in \mc{C}^*$.
\end{proof}

We are now ready to describe the isotropy at the dense subset of principal characters  in $\msc{G}^{(0)}$.

\begin{lem}\label{lem:categoryIsotropy}
For each $c \in \mc{C}$, the isotropy at $\chi_c$ is \[\msc{G}^{\chi_c}_{\chi_c} = \{[cuc^{-1}, \chi_c]\; | \; u \in \ssrc{c}\mc{C}^*\ssrc{c} \}.\]	
\end{lem}
\begin{proof}
Let $e = \ssrc{c}$. Then
	\begin{align*}
\msc{G}_{\chi_e}^{\chi_e} &= \{ [u,\chi_e] \; | \; u \in \mc{I}_l, u^{-1}u = e = uu^{-1} \} \\
& = \{ [u,\chi_e] \; | \; u \in \mc{C}^*, u^{-1}u = e = uu^{-1} \}, 
\end{align*}
	where the first equality is from \lemref{firstEqualityLem} and the second equality holds by  \lemref{secondEqualityLem}. We then have $\msc{G}_{\chi_c}^{\chi_c} = [c,\chi_e] \msc{G}_{\chi_e}^{\chi_e} [c,\chi_e]^{-1} = \{[cuc^{-1}, \chi_c]\; | \; u \in \ssrc{c}\mc{C}^*\ssrc{c} \}$.
\end{proof}

We use the above to show that the subgroupoid $\mc{C}^* \ltimes \Omega$ constrains the  isotropy in $\msc{G}$.

\begin{prop}\label{pro:categoryreltopp}
	The groupoid $\msc{G} = \overline{\mc{I}}_l \ltimes \Omega$ is  topologically principal relative to $\mc{C}^* \ltimes \Omega$.
\end{prop}
\begin{proof}
	Given $c \in \mc{C}$ the subset $[c,\Omega(c^{-1}c)]$ of $\msc{G}$ is an open bisection with $\ssrc{[c,\Omega(c^{-1}c)]} = \Omega(c^{-1}c)$ and $\ttrg{[c,\Omega(c^{-1}c)]} = \Omega(cc^{-1})$. If we let $\Gamma := [c,\Omega(c^{-1}c)]^{-1}$, then $\Gamma \msc{G}^{\chi_c}_{\chi_c} \Gamma^{-1} = \msc{G}^{\chi_{\ssrc{c}}}_{\chi_{\ssrc{c}}} \subseteq \mc{C}^* \ltimes \Omega$; that is, the isotropy at each element of  $X:=\{ \chi_c \; | \; c \in \mc{C} \}$  may be conjugated into $\mc{C}^* \ltimes \Omega$ by an open bisection. Since $X$ is dense by \lemref{semicharactersDenseLem}, $\msc{G}$ is topologically principal relative to $\mc{C}^* \ltimes \Omega$.
\end{proof}

An application of  \thmref{thm:groupoidUniquenessThm} now yields a faithfulness criterion
for representations of $C^*_r(\mc{C})$.

\begin{thm}\label{lcscUniquenessThm}
Let $\mc{C}$ be a left cancellative small category such that $\msc{G}:= \overline{\mc{I}_l} \ltimes \Omega$ is  Hausdorff.
A representation  of  $C_r^*(\mc{C}) \cong C^*_r(\overline{\mc{I}_l} \ltimes \Omega) $ 
is faithful if and only if its restriction  to the C*-subalgebra $C^*_r(\mc{C}^* \ltimes \Omega)$ is faithful. 
\end{thm} 
Under an amenability assumption this can be turned into a uniqueness result.
\begin{cor} Let $\mc{C}$ be a left cancellative small category such that $\msc{G}:= \overline{\mc{I}_l} \ltimes \Omega$ is Hausdorff.
If the canonical map $C^*(\overline{\mc{I}_l} \ltimes \Omega) \rightarrow C^*_r(\overline{\mc{I}_l} \ltimes \Omega)$ is an isomorphism, e.g. if $\overline{\mc{I}_l} \ltimes \Omega$ is amenable, then 
the C*-algebra generated by a representation of $C^*(\overline{\mc{I}_l} \ltimes \Omega)$ that is faithful on  $C^*(\mc{C}^* \ltimes \Omega)$ is canonically unique. 
\end{cor}
\begin{proof}
    By \thmref{lcscUniquenessThm}, a representation of $C^*(\overline{\mc{I}_l} \ltimes \Omega)$ that is faithful on  $C^*(\mc{C}^* \ltimes \Omega)$ is faithful.
\end{proof}

\begin{rem}
	When the LCSC $\mathcal{C}$ is a monoid that embeds in a group $G$  the groupoid $\overline{\mc{I}_l} \ltimes \Omega$ is isomorphic to the partial transformation groupoid of a partial action of $G$ on $\Omega$, and  \thmref{lcscUniquenessThm} recovers \cite[Theorem~5.1]{LS}. Since  there are left cancellative monoids that do not embed in a group and for which the groupoid $\overline{\mc{I}_l} \ltimes \Omega$ is Hausdorff, our result extends  \cite[Theorem~5.1]{LS} non-trivially even for monoids. A few concrete examples are discussed in the next section. 
\end{rem}

\section{Examples from integer arithmetic}\label{sec:integerarith}
  The purpose of this final section is to give a concrete application of \thmref{lcscUniquenessThm} to a monoid that does not embed in a group, so that \cite[Theorem 5.1]{LS} does not apply, and for which our `ideal-detecting subalgebra'  is strictly smaller than the one obtained from the interior of the isotropy as in \cite{brown2015cartan}.
  The monoid arises in the study of phase transitions of C*-dynamical systems of arithmetic origin \cites{Laca-Schulz23}.

Fix $n\in \nx$. Then $\N^\times$ acts on $\Z / n \Z$ by right multiplication, $[k] \cdot a := [ka]$ for $a \in \nx$ where $[k]= k+ n\Z \in \Z / n \Z$  denotes the equivalence class of $k$ modulo $n$. We consider the semidirect product  $\nxxzn$, where the operation is defined by 
\[
(a,[k])\cdot (b,[l]) = (ab,[kb+l])\qquad a,b \in \nx, \ k,l \in \Z.
\]

\begin{prop}(cf. \cite[Proposition 7.2]{Laca-Schulz23})
\label{pro:examplesproperties}
	The monoid $\nxxzn$  is left cancellative 
	but   not right cancellative, hence not embeddable in a group. Moreover, 
 \begin{equation}\label{eqn:P2rightLCM}
 (a,[k]) (\nxxzn) \cap (b,[l])(\nxxzn) = (\operatorname{lcm}(a,b),[0])(\nxxzn),
 \end{equation}
 so $\nxxzn$ is also right 
		 LCM. The associated groupoid $\msc{G}$  is Hausdorff and second-countable.

\begin{proof}   If $(ab,[kb + l]) = (a,[k])(b,[l]) = (a,[k])(c,[m]) = (ac,[kc + m])$, then $b = c$ and $[kb+l] = [kc +m]$.
  This implies $ [kb] = [kc]$, and hence $[l] =[m]$  in $\Z / n\Z$, hence $\nxxzn$ is left cancellative.

	 In order to see that right cancellation fails, consider for example that $(1,[2])\neq (1,[1])$, but right-multiplication by $(n,[0])$ gives $(1,[2])(n,[0]) = (n,[2n+0])= (n,[0]) = (n,[n+0])=(1,[1])(n,[0])$.
  
It is straightforward to verify \eqref{eqn:P2rightLCM} using the composition rule,  see e.g. \cite{Laca-Schulz23}, proving the right LCM property. Notice that the principal right ideal  $(a,[k])(\nxxzn)$ does not depend on $[k]$.
		
 Since $\nxxzn$ is countable,  second countability of $\msc{G}$ follows from \lemref{semicharactersDenseLem} and its proof. The only thing left to show is that  $\msc{G}$  is Hausdorff. Recall from \cite[Definition~3.2.]{SpielbergSmallCats} that a left cancellative small category $\mc{C}$ is \textit{finitely aligned}  if for every $a,b \in \mc{C}$ there exists a finite set $F \subseteq \mc{C}$ such $a\mc{C} \cap b\mc{C} = \bigcup_{c \in F} c\mc{C}$. Right LCM monoids are obviously finitely aligned as small categories, with $F = \{ c\}$. By \cite[Corollary~4.2.]{LiGarsideCats}, to prove $\msc{G}$ is Hausdorff it suffices to show that for all $(a,[k]),(b,[l]) \in  \nxxzn$ the set 
  \[
  X:= \big\{(c,[m]) \in \nxxzn \; | \; (a,[k])(c,[m]) = (b,[l])(c,[m]) \big\} 
  \]
  equals $(d,[t])\nxxzn$ for some $(d,[t]) \in \nxxzn$. Let $k,l,m$ be positive integers smaller than $n$, let $a,b,c \in \nx$ be arbitrary, and suppose that $(a,[k])(c,[m]) = (b,[l])(c,[m])$. Then $ (ac,[kc +m]) = ( bc,[lc + m])$, so we have $a = b$ because $\nx$ is right cancellative. Moreover $[kc + m] = [lc +m] \iff [kc] = [lc]$ so the choice of $[m]$ is arbitrary and it follows that $X = \{ (c,[m]) \in P \; | \; [kc] = [lc] \}$. We claim that 
	\[
   X = \big\{ (c,[m]) \in \nxxzn \; | \; [kc] = [lc] \big\} = \left(\frac{\operatorname{lcm}(k-l, n)}{k-l},[0] \right) \nxxzn.
   \]
			First let
   $(\frac{\operatorname{lcm}(k-l, n)}{k-l} \cdot t, [p])$ be an element of the right hand side; then $$\left[k\cdot \frac{\operatorname{lcm}(k-l, n)}{k-l} \cdot t\right] - \left[l\cdot \frac{\operatorname{lcm}(k-l, n)}{k-l} \cdot t\right] = \left[\operatorname{lcm}(k-l, n) \cdot t\right] = [0].$$
			So $(\frac{\operatorname{lcm}(k-l, n)}{k-l} \cdot t, [p]) \in X$. Conversely suppose $(c,[m]) \in X$. Then $[kc] = [lc]$ implies $n$ divides $(k-l)c$ hence $\operatorname{lcm}(k-l,n)$ divides $(k-l)c$ and there is a $t$ such that $(k-l)c = \operatorname{lcm}(k-l,n)t$ hence $c = \frac{\operatorname{lcm}(k-l,n)}{k-l}t$. It follows that $(c,[m]) = (\frac{\operatorname{lcm}(k-l,n)}{k-l}, [0])([m],t)$ as desired. This completes the proof.
\end{proof}
\end{prop}

\thmref{lcscUniquenessThm} yields the following faithfulness and uniqueness result.

\begin{cor}\label{cor:critforsemid}
   A representation of $C^*(\nxxzn)$ is faithful if and only if its restriction to the crossed product 
    $C(\Omega) \rtimes (\{1\} \ltimes\Z/n\Z)$ is faithful.
\end{cor}
\begin{proof}  It is easy to verify that $(\nxxzn)^*= \{1\} \rtimes\Z/n\Z$. 
    The result now follows from  \thmref{lcscUniquenessThm}, whose assumptions are satisfied by 
  \proref{pro:examplesproperties} and \proref{pro:categoryreltopp}. 
\end{proof}
Next we aim to show that the conditions for detecting ideals from \cite{brown2015cartan} and our generalization to that  from \cite{LS} are effectively different.  In order to do  this we first need to compute  the interior of the isotropy.

\begin{prop}\label{interiorOfIsoISActionProp}
	Let $S \curvearrowright \Omega$ be an action of an inverse semigroup on a topological space. The interior of the isotropy of the action groupoid $\msc{G} := S \ltimes \Omega$ consists of all pairs $[s,x] \in \msc{G}$ such that there exists an open neighborhood $U$ of $x$ for which $s\cdot y = y$ for all $y \in U$.
\end{prop}
\begin{proof}
	A basis for the topology of $S \ltimes \Omega$ is given by the subsets $[s,U] = \{ [s,\chi] \; | \; \chi \in U\}$ for $s \in S$ and $U$ an open subset of $D_{s^{-1}s}$. It follows that a point $[s,x] \in S \ltimes \Omega$ is in the interior of the isotropy if and only if we can choose $U\subseteq D_{s^{-1}s}$ such that $s$ fixes every element of $U$.
\end{proof}

\begin{prop}\label{inverseSemigroupActionGroupoidIsoProp}
	Let $\mc{C}$ be a left cancellative small category. The interior of the isotropy of $\msc{G}:= \overline{I_l} \ltimes \Omega$ is the collection of all pairs $[s,\chi]$ such that there exists an idempotent $e \in E(\overline{\mc{I}_l})$ satisfying \begin{itemize}
	    \item $e \leq s^{-1}s$
     \item $\chi(e) = 1$
     \item for each $c \in \operatorname{dom}(e) \subseteq \mc{C}$ there exists $u \in \mc{C}^*$ such that  $[s,\chi_c] = [c uc^{-1},\chi_c]$.
     \end{itemize}
\end{prop}
\begin{proof}
	Let $[s,\chi] \in \operatorname{Iso}(\msc{G})^\circ$; then there is a basic open neighborhood $\Omega(e) = \{\psi \in \Omega \; | \; \psi(e) = 1 \}$ of $\chi$ such that $s\cdot \psi = \psi$ for all $\psi \in \Omega(e)$. In particular, if $c \in \dom{e}$, then $\chi_c \in \Omega(e)$ and  $s$ fixes $\chi_c$ so $[s,\chi_c] = [c uc^{-1},\chi_c]$ for some $u \in \mc{C}^*$.
	Conversely, if $[s,\chi] \in \msc{G}$ and there is an idempotent $e \leq s^{-1}s$ such that $\chi(e) = 1$ and $[s,\chi_c] \in [c \mc{C}^* c^{-1}, \chi_c]$, then $s$ fixes a dense subset of $\Omega(e)$ so $s$ fixes every $\psi \in \Omega(e)$ by continuity.
\end{proof}

Next we  show that the action groupoid of invertible elements acting on the unit space is a more efficient way to detect ideals of $\Toepr(\nxxzn)$ than the interior of the isotropy.

\begin{prop}\label{pro:P2} 
   $ (\nxxzn)^* \ltimes \Omega  \subsetneq \operatorname{Iso} (\msc{G})^\circ $.
\end{prop}
\begin{proof}
First we show that $(\nxxzn)^* \ltimes \Omega\subset \operatorname{Iso} (\msc{G})^\circ $. Suppose $(1,[k]) \in (\nxxzn)^*$ and let  $(a,[m])\in \nxxzn$.   Then $(1,[k])\cdot \chi_{(a,[m])} = \chi_{(a,[ka+m])} =\chi_{(a,[m])}$ because $(a,[ka+m]) =(a,[m])$ in $\nxxzn$. Since the characters  $\chi_{(a,[m])}$ are dense in $\Omega$ by \lemref{semicharactersDenseLem},  we see that $(\nxxzn)^* \ltimes \Omega$ is contained in the isotropy at every point. But $(\nxxzn)^* \ltimes \Omega$ is open, hence it is actually  contained in the interior of the isotropy.

 Next define an element $s$ of the left inverse hull by
\begin{align*} s:= (n,[0])(1,[1])(n,[0])^{-1}: (n,[0])(\nxxzn) &  \xlongrightarrow{\cong} (n,[0])(\nxxzn) \\
(n,[0])(a,[m]) &\longmapsto (n,[0])(a,[a+m]).
\end{align*}
We will show that 
\[ [s, \chi_{(n,[0])}] \in \operatorname{Iso}({(\nxxzn)\ltimes \Omega_2})^\circ \setminus (\nxxzn)^* \ltimes \Omega.
\]

 We show first that the element $[s, \chi_{(n,[0])}]$ is in the interior of the isotropy of $\msc{G}$. First note that points of the form $\chi_{(n,[0])(b,[m])} = \chi_{(nb,[m])}$, with $b \in \nx$ and $0 \leq m \leq n-1$, are dense in $V := \Omega((n,[0])(\nxxzn))$. For every such point we have $s \cdot \chi_{(nb,[m])} = \chi_{(nb,[m+b])} = \chi_{(nb,[m])}$ where the last equality is due to the fact that for each $e \in E(\overline{\mc{I}_l})$ we have $(nb,[m]) \in \dom{e} \iff (nb, [k]) \in \dom{e}$ for all $0 \leq k \leq m-1$. Thus $s$ fixes a dense set of points in $V$ and must fix every point in $V$ by continuity. Then $[s, \chi_{(n,[0])}] \in \operatorname{Iso}(\msc{G})^\circ$ by Proposition \ref{interiorOfIsoISActionProp}.
  
   To conclude that $[s, \chi_{(n,[0])}] \notin (\nxxzn)^*\ltimes \Omega$  we need to show that for all $e \leq s^{-1}s$ such that $(n,[0]) \in \dom{e}$ and $u = (1,[k]) \in (\nxxzn)^*$ we have $se \neq ue$. Suppose for contradiction that $e \in E$ and $u = (1,[k])$ are as above and that $se = ue$ and let $(na,[m]) \in  \operatorname{dom}(e) \subseteq (n,[0])(\nxxzn)$ be arbitrary. Then $se = ue$ implies that $(na,[m+a]) = s(na,[m]) = (1,[k])(na,[m]) = (na,[m])$ so $[m+a] = [m]$ and $n$ must divide $a$. But then $\dom{e} \subseteq (n^2,[0])(\nxxzn)$ which  contradicts   the assumption that $(n,[0]) \in \dom{e}$. 
\end{proof}

\begin{remark}
 {From the proof of \proref{pro:P2} we see that $(\nxxzn)^* \ltimes \Omega$ is an open subgroupoid of the interior of the isotropy. Moreover, at the canonical dense subset of characters $\chi_{(a,[m])}$ the isotropy of this subgroupoid coincides with that of $\msc{G}$. Thus, our  \corref{cor:critforsemid} already follows from Starling's improvement \cite[Theorem 2.1(ii)]{Starling} of the criteria from \cite{brown2015cartan}.}

    A similar result to \proref{pro:P2} can be proved about the monoid  $\nxxz$, which is group embeddable, and about $\nxxqz$, which is not.
\end{remark}
In contrast to the above applications we now consider the submonoid $\Z \rtimes \nx$ of the group $\Q \rtimes \Q_+^*$, where the operation is given by $(k,a)(l,b) = (k + al, ab)$ and the  set of invertible elements is $(\Z \rtimes \nx)^* = \Z \rtimes \{1\}$. This monoid is the opposite of $\nxxz$ and corresponds to the additive boundary quotient in \cite{LacaRaeburn}. It has a finer principal ideal structure, see 
    \cite[Section 5]{LacaRaeburn} and the isotropy at principal characters is consequently  smaller. As a result, we see that ideal detection is more efficient using the interior of the isotropy, which in this case is simply the unit space.
\begin{prop}
      Let $\Omega := \msc{G}^{(0)}$ be the unit space of the groupoid $\msc{G}$ associated to the monoid $\Z \rtimes \nx$. Then  $\operatorname{Iso}(\msc{G})^\circ \subsetneq (\Z \rtimes \nx)^* \ltimes \Omega$.
\begin{proof}
 The action of $(\Z \rtimes \nx)^*$ on $\Omega$ is topologically free by \cite[Lemma 1.6]{ToeLaNe}, and by \cite[Theorem 5.9]{LS} so is the partial action of $\Q_+^* \rtimes \Q$. Thus the interior of the isotropy is the unit space, which is properly contained in the action subgroupoid  $(\Z \rtimes \nx)^* \ltimes \Omega$.     For example, the element $((1,1),\chi_{(0,2)})  \in (\Z \rtimes \nx)^* \ltimes \Omega$ is not in $\operatorname{Iso}(\msc{G})$ because $(1,1) \cdot \nolinebreak \chi_{(0,2)} = \chi_{(1,2)}\neq\chi_{(0,2)} $.
    \end{proof}
   \end{prop} 
    
   \bigskip
   \section*{Disclosure and data availability statement} 
   The authors have no relevant financial or non-financial interests to disclose. 
   There is no data associated to this paper.

\end{document}